\documentclass[11pt]{article}

\usepackage{amssymb,amsmath,amsthm,url}
\usepackage{hyperref}
\usepackage{color}
\usepackage[boxed]{algorithm}
\newtheorem{theorem}{Theorem}[section]
\newtheorem{lemma}[theorem]{Lemma}
\newtheorem{proposition}[theorem]{Proposition}
\newtheorem{corollary}[theorem]{Corollary}

\newtheorem{definition}[theorem]{Definition}

\newcommand{\beq}[1]{\begin{equation}\label{#1}}
\newcommand{\enq}[0]{\end{equation}}

\newcommand{\C}[2]{{{#1}\choose{{#2}}}}

\def\E{\mathop{\mathbb E}}

\newcommand{\A}{\mathcal A}
\newcommand{\B}{\mathcal B}
\newcommand{\M}{\mathcal M}

\newcommand{\N}{\mathcal N}

\newcommand{\remove}[1]{}

\begin{document}
\title{Thresholds and expectation-thresholds of monotone properties with small minterms}
\author{Ehud Friedgut \and Jeff Kahn  \and Clara Shikhelman}

\setlength{\parskip}{1ex plus 0.5ex minus 0.2ex}

\maketitle

\begin{abstract}
Let $N$ be a finite set, let $p \in (0,1)$, and let $N_p$ denote a random binomial subset of $N$ where every element of $N$ is taken to belong to the subset independently with probability $p$ . This defines a product measure $\mu_p$ on the power set of $N$, where for $\A \subseteq 2^N$ $\mu_p(\A) := Pr[N_p \in \A]$.

In this paper we study upward-closed families $\A$ for which all minimal sets in $\A$ have size at most $k$, for some positive integer $k$. We prove that for such a family $\mu_p(\A) / p^k $ is a decreasing function, which implies a uniform bound on the coarseness of the thresholds of such families. 

We also prove a structure theorem which enables one to identify in $\A$ either a substantial subfamily $\A_0$ for which the first moment method gives a good approximation of its measure, or a subfamily which can be well approximated by a family with all minimal sets of size strictly smaller than $k$.

Finally, we relate the (fractional) expectation threshold and the probability threshold of such a family, using duality of linear programming. This is related to the threshold conjecture of \cite{KK}.

\end{abstract}\section{Introduction}
 One of the fundamental phenomena in random graph theory is that of thresholds of monotone properties. This dates back to the seminal papers of Erd\H{o}s and Renyi \cite{ER59,ER60} who defined the notion of thresholds and discovered that for many interesting graph properties the probability of the property appearing in the random binomial graph $G(n,p)$, for large $n$, behaves much like a step function of the edge probability $p$, increasing from 0 to 1 abruptly as $p$ is varied slightly.
The study of thresholds of random structures in general, and in random graphs specifically, has been a thriving area ever since, and thousands of papers have covered related problems. Bollob\`as and Thomason \cite{BT} showed that every monotone property of sets has a threshold function and, using the Kruskal-Katona theorem, gave optimal quantification of such thresholds. In \cite{FK} it was observed that the KKL theorem \cite{KKL}, and its extension in \cite{BKKKL} imply sharp thresholds for properties which are symmetric under the action of a group on the elements of the ground set, in particular for graph properties.

For most interesting families of graph properties the threshold function $p(n)$ tends to zero as $n$ tends to infinity. In this case it is of interest to study the sharpness of the threshold. Fixing a graph property $\A$, and a parameter $\epsilon$, one may ask what is the width of the interval of values of $p$ in which the probability of $G(n,p)$ having property $\A$ climbs from $\epsilon$ to $1-\epsilon$. The scale in which this width is measured is with respect to the value of $p$ for which the probability of $G(n,p) \in \A$ is, say, 1/2. For a series of properties $\A_n$, of graphs on $n$ vertices,  we will say that the threshold is sharp if the ratio between the width of the threshold interval and the critical $p$ tends to 0. We will shortly give a more precise definition of sharp thresholds in a more general setting.

In his Ph.D. thesis, the first author \cite{F99} gave a necessary condition for monotone graph properties to have a sharp threshold.  Roughly speaking, if a property does not have a sharp threshold it must be well approximable by a {\em local property} (e.g. containing a triangle), as opposed to properties that are {\em global} (e.g. connectivity) and cannot be well approximated by the property of containing a subgraph from a fixed given list.  In the appendix to \cite{F99} Bourgain proved a similar statement, with a slightly weaker conclusion, in a much more general setting, without the assumption of symmetry. In a recent paper Hatami, \cite{Hatami} gives a common generalization of these two results. 

Returning to the question of thresholds of local properties, the appearance of any fixed subgraph in $G(n,p)$ has a coarse threshold, and this is well understood, see Bollob\'as' paper \cite{B} for a complete description.  Roughly speaking, if a fixed graph $H$ is strictly balanced then the number of copies of $H$ in $G(n,p)$ will be approximately Poisson, and the governing parameter, the expectation of the random variable, will be of order $p^{|E(H)|}n^{|V(H)|}$, which varies smoothly with $p$: when $p$ is multiplied by a constant $c$, the expectation changes by a factor of $c^{|E(H)|}$. When $H$ is not balanced the situation is only slightly more complicated, and appearance of copies of $H$ in $G(n,p)$ can be understood by studying the appearances of the densest subgraphs of $H$. Intuitively, if $H$ has a subgraph $H'$ which is much denser than $H$, then every copy of $H'$ that appears in $G(n,p)$ is extremely likely to be contained in many copies of $H$. 
Kahn and Kalai, \cite{KK}  have a far reaching conjecture as to the generalization of this to general monotone properties. In the graph setting their conjecture regards families defined by graphs with size which is not fixed (i.e. such as  Hamiltonian cycles). In a nutshell, they conjecture that for such families there is at most a logarithmic gap between the threshold probability for the appearance of a graph from the family, and the probability at which the expectation is constant (once again, taking into account the densest subgraphs).

The basic question which led to the writing of this paper was: how specific is this behavior to graphs? The proofs of this behavior use the symmetry of graphs very strongly, yet it seemed possible that something similar should hold also for properties of random binomial subsets of a ground set without any symmetry assumptions.
This would imply a converse to the main theorems of \cite{F99} and its appendix: not only does a non-sharp threshold imply that the property in question has local nature,  but also any property determined by small minimal sets has a non-sharp threshold. 

We will prove in this paper that this indeed is the case.   

\section{Setting and main results}

 Let $[n] $ denote the set $\{0,1,\ldots,n\}$, and let $[n]_p$ denote a random subset of $[n]$, where each element is chosen independently with probability $p$.
A family of sets $\A$ is called {\em monotone} if whenever $A \in \A$ and $A \subset B$, then $B \in \A$. For such a family, a set which is minimal with respect to inclusion is called a {\em minterm}.
For $\A$, a family of subsets of $[n]$ and $p \in [0,1]$ we define $\mu(\A,p)$ to be the probability that $[n]_p \in \A$. Note that if $\A$ is monotone then this function is monotone in $p$. We will also use the notation $\mu_p$ to denote the measure $\mu( \cdot, p)$.

For a fixed non-trivial monotone family $\A$ and any $x \in [0,1]$ we define $p_x$ to be the unique number such that $\mu(\A, p_x)=x$. For $0< \epsilon <1 $  define
$\delta_\epsilon(\A) = \frac{p_{1/2}-p_\epsilon}{p_{1/2}}$. The numerator, is the length of the {\em threshold interval} in which the probability of $\A$ climbs from $\epsilon$ to $1/2$. The denominator, $p_{1/2}$, supplies the correct yardstick with which to measure this length. The slower $\delta_\epsilon (\A)$ tends to 1 as $\epsilon$ tends to 0, the sharper the threshold is (in other words, the threshold interval is small), the faster it tends to 1 the coarser the threshold is. (Note that it also would make sense to study the interval $[p_\epsilon,  p_{1-\epsilon}]$, however our choice gives a neater normalization, bounding $\delta_\epsilon (\A)$ between 0 and 1).

\begin{theorem}\label{thm:decrease}
Let $\A$ be a monotone family of subsets of $[n]$, with all minterms of size at most $k$. Then the function 
$\frac{\mu(\A,p)}{p^k}$ is monotone decreasing. Consequently 
$$\delta_\epsilon(\A)\ge 1-(2\epsilon)^\frac{1}{k}.$$
\end{theorem}

The simple derivation of this theorem from the Margulis-Russo lemma was pointed out to us by Oliver Riordan.
Note that the theorem is tight, e.g., for a family with a single minterm.
We present the proof of Theorem \ref{thm:decrease} in Section \ref{sec:lowerbound} below.
An upper bound on $\delta_{\epsilon}$ will follow from a different approach which we present in Section
\ref{sec:LP}:
\begin{theorem}\label{thm:upperbound}
$$1-2\epsilon \frac {(k-1)^{k-1}}{k^k} \ge\delta_\epsilon(\A) .$$
\end{theorem}
To present the results of Section \ref{sec:LP} we first need a definition. For a monotone family $\A \subseteq P([n])$ with a set of minterms $\M$ we define $\E_p(\A)$ to be the expected number of minterms of $\A$ that are contained  in the random set $[n]_p$. $\E_p(\A)$ is the expectation of the function $f(A) = \sum_{M \in \M} {\bf 1}_{A \subset M}$, which takes positive integer values on all $A$ in $\A$, and consequently gives an upper bound on $\mu(p,\A)$.
This upper bound can be tightened by a fractional version, which we call the {\em fractional expectation } of $\A$.
\begin{definition}
For a monotone family $\A \subset P([n])$, with a set of minterms $\M$, we define the fractional expectation of $\A$ with respect to $\mu_p$ to be
$$
{\E}^*_p (\A) =\min \sum \beta(B) p^{|B|}, 
$$
where the minimum is taken over all functions $\beta : P([n]) \rightarrow \mathbf{R}^{\geq 0}$ such that 
$$
\sum_{B \subseteq A} \beta(B) \ge 1 \text { for all } A \in \M.
$$
\end{definition}
If $\beta$ is a function for which the minimum in the definition of the fractional expectation is achieved, then ${\E}^*(\A)$ is  the expectation of the function  
$g(A) = \sum_{B} \beta(B) {\bf 1}_{B \subseteq A}$, which is non negative and assumes values at least 1 on all $A$ in $\A$, hence this too gives a (better) upper bound on $\mu(\A)$, namely
$$
{\E}_p(\A) \ge {\E}^*_p(\A) \ge \mu_p(\A).
$$
The main result of Section \ref{sec:LP} is that for monotone families with minterms of bounded size this bound is not too far off mark.
\begin{theorem}\label{thm:LP}
Let $\A$ be a monotone family with all minterms of size at most $k$. Then for any $\alpha >0$
$$
{\E}^*_p(\A) \ge \mu_p(\A) \ge {\E}^*_{\alpha p}(\A)(1+\alpha)^{-k}.
$$
\end{theorem}
As a corollary we deduce a special case of the expectation-threshold conjecture of \cite{KK}.
\begin{corollary}
If $ {\E}^*_p(\mathcal{A}) =1$ then $\mu_{kp}(\mathcal{A}) > 1/e$.
\end{corollary}
Note that for families with minters of size at most $k$ the expectation threshold and the fractional expectation threshold differ only by a constant factor of $2^k$. Talagrand conjectures in \cite{Talagrand} that the gap between the thresholds for families with minterms of size at most $k$, is of order at most $\log(k)$. It is interesting whether the methods in this paper can be further pushed to yield this result.

In Section \ref{sec:structure} we approach the question of understanding the threshold behavior of a monotone family 
$\A \subseteq P([n])$ via the parameter $\E_p(\A)$, the expected number of minterms of $\A$ in a random set $[n]_p$.  
If we have good control over the second moment of this random variable, the expectation gives us a good indication as to the probability $\mu_p(\A)$. An example of this setting is, say, the family of all subgraphs of $K_m$ that contain a copy of $K_4$, when $p=\Theta(m^{-2/3})$. It is easy to verify that in this case the expected number of minterms (i.e. the expected number of copies of $K_4$ in $G(m,p)$) is $\Theta(1)$, whereas the variance is also of this order of magnitude, which enables us to get an effective lower bound on the measure of the family using the Payley-Zygmond bound 
$$
Pr [Z >0] \ge \frac{E[Z^2]}{(E[Z])^2},
$$
which holds for any non-negative random variable $Z$.
On the other hand, consider the example where the minterms are all subgraphs of $K_m$ containing a copy of "$K_4$ with a tail'', a graph consisting of $K_4$ with a fifth vertex connected to precisely one of the four. We again set 
$p=m^{-2/3}$ and a moment of thought shows that although this family is properly contained in the previous one, the measure of their symmetric difference is negligible, as any copy of $K_4$ that appears in $G(m,m^{-2/3})$ is overwhelmingly likely to have many tails. This is reflected by the fact that the expectation now is huge rather than constant. In this case one has to realize that the tail connected to $K_4$ is a red herring, and proper analysis can, and should, focus on  the previous family.

Such examples are almost canonical in any introductory course to random graphs. Our main theorem in 
Section \ref{sec:structure} guarantees that something similar to one of these two case should hold in any family $\A$ defined by minterms of bounded size $k$. Either there is a substantial subfamily $\B$ for which the first and second moments are well behaved, or a substantial subfamily $\B$ that may be approximated by a different family with minterms of size strictly smaller than $k$.
This structure theorem then allows us to deduce a theorem quite similar to Theorem \ref{thm:decrease}, with a slightly worse rate of decay (e.g. $\mu_{p/2}(\A) \ge \mu_p(\A)/k8^k$, as opposed to the truth which is 
$\mu_{p/2}(\A) \ge \mu_p(\A)/2^k$.)

\section{The Margulis-Russo lemma, and proof of Theorem \ref{thm:decrease}}\label{sec:lowerbound}
For a monotone family $\A$, the Margulis-Russo lemma, (\cite{M}, \cite{R}) relates the derivative of $\mu(\A,p)$ with respect to $p$ with the {\em edge boundary} of $\A$. If $A \in \A$, but $(A \setminus {a})$ is not in $\A$ we say that 
$a$ is a pivotal element of $A$ and that there is a boundary edge  ''leaving $A$ in the direction of $a$". Let $Piv(A)$ denote the number of pivotal elements in $A$ (which is necessarily 0 if $A \not \in \A$).  Let $A$ be a random set chosen according to $\mu(p)$, then $Piv(A)$ is a random variable, and its expectation is a measure of the size of the boundary of $\A$. As the lemma below shows, this is a parameter intimately correlated with the threshold behavior of $\A$.

\begin{lemma}\label{MR}[Margulis,Russo]
$$ p \frac{d \mu(\A,p)}{d p} = \E[Piv(A)].$$
\end{lemma}
This lemma, which is so simple to state, and, as we shall see shortly, very easy to prove, is extremely useful.
See, e.g., \cite{FK},\cite{F99},\cite{BR} . It is not surprising that this lemma is relevant when studying thresholds, as the expression on the right hand side is clearly related to the ratio between the width of the threshold interval and the value of $p$ within the interval (it is an approximation of its reciprocal). Both Russo and Margulis proved this lemma by induction on $n$, the size of the ground set from which $A$ is chosen. For the sake of being self contained we present a different proof, which is well known folklore, perhaps due to Gil Kalai.

{\bf Proof of Lemma \ref{MR}:} 
Let $\A$ be a monotone family of subsets of $[n]$, and for some fixed $(p_1,p_2,\ldots,p_n)$ consider the following product measure $\mu_{(p_1,p_2,\ldots,p_n)}$ on 
$P([n])$. The measure of a set $A$ is $\prod_{i \in A} p_i \prod_{j \not \in A} (1-p_j) $. For $i \in [n]$ and a random set $A$ chosen according to $\mu_{(p_1,p_2,\ldots,p_n)}$ let $a_i$ denote the probability that ($A \in \A$ and $(A \setminus \{i\}) \in \A$), and let $b_i$ denote the probability that $i$ is pivotal in $A \cup \{i\}$, i.e.  ($(A \cup \{i\}) \in \A$ and $(A \setminus \{i\}) \not \in \A$). This means that the probability that $i$ is pivotal is $p_i b_i$. Recalling that $\A$ is monotone, for all $i$ we have $\mu_{(p_1,p_2,\ldots,p_n)}(A) = a_i + p_ib_i$, and $\E[Piv(A)] = \sum_i p_ib_i.$

We now let all of the $p_i$ depend on a common parameter $p$ in the following trivial way: $p_i(p)=p$, and note that the resulting measure is $\mu_p$. The Margulis-Russo formula now follows from a simple application of the chain rule.
$$
\frac{d \mu(\A,p)}{d p} = \sum_i \frac{\partial \mu_{(p_1,p_2,\ldots,p_n)}(\A)}{\partial p_i} \cdot \frac {d p_i}{dp}= \sum b_i =\frac{ \E[Piv(A)]} {p} .
$$
\qed

The fact that Theorem \ref{thm:decrease} follows immediately from the Margulis-Russo formula, as pointed out to us by Oliver Riordan, is yet another example of how useful this result is.

{\bf Proof of Theorem \ref{thm:decrease}:} Let $\A$ be a monotone family with all minterms of size at most $k$. Let $A$ be a random set chosen according to $\mu_p$. Note that $A$ can never have more than $k$ pivotal elements, and that if $A \not \in \A$ then, by definition there are no pivotal elements. Therefor 
$$
  \E[Piv(A)] \leq k\cdot\mu_p(A). 
$$ 
Using this in conjunction with the Margulis-Russo formula, and deriving with respect to $p$ gives
$$
\left(\frac{\mu_p(\A)}{p^k}\right)' = \frac{\mu'p^k - kp^{k-1}\mu}{p^{2k}} \leq 0.
$$
\qed

\section{A structure theorem for monotone families with small minterms.}\label{sec:structure}
We begin with some notation.
Let $\A$ be a monotone family and let $\M(A)$ be the set of its minterms. Throughout this section we will assume that all minterms are of size at most $k$. Let $X^{\A}$ be the random variable that counts the number of minterms of $\A$ in $[n]_p$. For a set $V\subset [n]$ we will define the family of its  $m-$supplements with respect to $\A$ to be the family of sets $W$ of size $m$ whose disjoint
union  with $V$ form a minterm, namely
$N^{m}_{\A}(V)=\{W \subset [n] \mbox{ s.t. } |W|=m \mbox{ and } \exists M \in \M(\A) \mbox{ s.t. } W\uplus  V= M\}$.

We say that $\A$ is \textit{tame} with respect to $p$ if for any $1\leq m\leq k-1$ and for any $V\subset[n]$ one has $|N^{m}_{\B}(V)|<p^{-m} $. We say that $\B$ is a \textit{tame $m$-approximation} of $\A$ at $p$, if there is a subfamily 
$\A'\subset \A$, such that for all minterms $B\in \M(B)$ it holds that the set $N_{\A'}^{m}(B)$ has size at least $p^{-m}$ and is tame.

The definition of a tame family is useful because it implies that for such a family the first moment bound on the measure is not too far from the truth. This is captured by the following lemma.
\begin{lemma} \label{lemma:tameFamily}
Let $\A$ be a tame family with respect to $p$ with minterms of size at most $k$. Then 
$$
\mu_{p}(\A)\ge \frac{\min \{\E_p[X^{\A}],1\}}{k2^k}
$$
\end{lemma}

\begin{proof}

We would like to use the Paley-Zygmund inequality to bound the probability from below:
\begin{equation} \label{eq:PZ}
\mu_p(\A)\geq \frac{\E^2[X^{\A}]}{\E[(X^{\A})^2]}
\end{equation}

Let us calculate the numerator and denominator separately.
Denoting $\M(m)= |\{M_i\in \M \mbox{ s.t. } |M_i|=m\}|$ it's easy to see that:
$$\E[X^\A]^2=\sum_{m,\ell=1}^k p^{m+l} \M(m)\M(l)$$

The denominator needs a bit more careful work. Remembering that $\M(\A)=\{M_i\}$ and that $X^\A$ is the random variable counting the number of minterms, we can define $X_i$ to be the indicator of $M_i$ and write $X^\A = \sum_i X_i$. With this we have:
\begin{align*}
\E[X^2]&=\E[\sum_{i,j}X_i X_j]
=\E[\sum_i X_i^2]+\E[\sum_{ \substack{ i\ne j \\ M_i \cap M_j \ne \varnothing }} X_i X_j] +\E[\sum_{M_i\cap M_j=\varnothing} X_iX_j]\\
\end{align*}
it's easy to see that:
\begin{align*}
\E[\sum_{M_i\cap M_j=\varnothing} X_iX_j]&\leq \sum_{m,l=1}^{k} p^{m+l}\M(m)\M(l)=\E[X^\A]^2
\end{align*}
and as the $X_i$s are indicators we get
$$
\E[\sum_i X_i^2]=\E[\sum_i X_i]=\E[X^\A]
$$
and so we are left with taking care of the second summand. 

Note that $\A$ is a tame family, thus for any $M_i$ and $1\leq m \leq k-1$ one has:
\begin{align*}
|\{ M_j \mbox{ s.t. } M_i\cap M_j\ne \varnothing, |M_j \cap M_i|=l \}|&=\sum_{V \subset M_i} N_\A^{l}(V)\leq 2^k p^{-l}
\end{align*}
With this we are ready to do the calculation. We break the sum into sums corresponding to the different sizes of minterms and supplements. 
\begin{align*}
\E[\sum_{ \substack{ i\ne j \\ M_i \cap M_j \ne \varnothing }} X_i X_j]&=\sum_{m=1}^{k}\sum_{i: |M_i|=m} (\sum_{j=1}^{k-1} \sum_{\substack{j:M_j\cap M_i\ne \varnothing \\ |M_j\setminus M_i|=l}}\E[ X_iX_j] )  \\
&=\sum_{m=1}^{k}\sum_{i: |M_i|=m} (\sum_{j=1}^{k-1} |\{ M_j \mbox{ s.t. } M_i\cap M_j \ne \varnothing, |M_j \setminus M_i|=l \}|  p^{m+l} ) \\
&\le \sum_{m=1}^{k}\sum_{i: |M_i|=m} (\sum_{j=1}^{k-1} 2^k  p^{m} )\\
&=(k-1)2^k\sum_{m=1}^{k}\sum_{i: |M_i|=m}p^m\\
&=(k-1)2^k\sum_{m=1}^{k}\M(m)p^m\\
&=(k-1)2^k\E[X^\A]
\end{align*}
Summing everything together we get that:
$$
\E[(X^\A)^2]\leq ((k-1)2^k+1)\E[X^\A]+\E[X^\A]^2
$$
And now plugging this in to Paley-Zygmund we get:
$$
\mu_{p}(\A)\ge \frac{\E[X^\A]^2}{\E[(X^\A)^2]}\geq \frac{\E[X^\A]^2}{((k-1)2^k+1)\E[X^\A]+\E[X^\A]^2}\geq \frac{\min\{\E[X^\A],1\}}{(k-1)2^k+2}
$$
For the simplicity of further calculations we can obviously write:
$$
\mu_{p}(\A)\ge \frac{\min\{\E[X^\A],1\}}{k2^k}
$$
as needed.
\end{proof}

\begin{corollary} \label{cor:TameApp}
Let $\A$ be a monotone family with minterms of size at most $k$. Then 
\begin{enumerate}
\item
If $\A$ is tame with respect to $p/2$ then $\mu_{p/2}(\A) \ge \mu_p(\A) / (k 2^{2k}).$
\item 
If $\B$ is a tame $m$-approximation of $\A$ at $p$ then $\mu_p(\A) \ge \mu_p(\B) / (m2^m)$

\end{enumerate}

\end{corollary}
Note that item (1) is weaker than what can be deduced using Theorem \ref{thm:decrease}, nonetheless we include it with its proof in order to demonstrate the information one can deduce from the structural approach.
\begin{proof}
To see (1) we only need to note that as the minterms of $\A$  are of size at most $k$ one has $\E_{p/2}[X^{\A}]\geq \frac{1}{2^k} \E_{p}[X^{\A}]$, together with the inequality on the first moment  $\E_{p}[X^{\A}]\geq \mu_{p}(\A)$ and Lemma \ref{lemma:tameFamily} we get
$$
\mu_{p/2}(\A)\ge \frac{\min\{\E_{p/2}[X^\A],1\}}{k2^k}\geq \frac{\min\{\E_{p}[X^\A],1\}}{k2^{2k}} \geq \frac{\mu_{p}(\A)}{k2^{2k}}
$$
as needed.

For (2) note that as any $B\in \M( \B)$ is a subset of a minterm of $\A$ the following inequality holds for any $p$:

$$
\mu_p(\A)\geq \mu_p(\B) \min_{B\in \B} \mu_p(N_{\A}^{m}(B))
$$
It is left to show that for any $B\in \B$ one has $ \mu_p(N_{\A}^{m}(B)) \geq \frac{1}{m2^m}$. $\B$ is an $m-$approximation so there is a family $\A'\subset \A$ for which $|N_{\A'}^{m}(B)|>p^{-m}$, furthermore each minterm in $N_{\A'}^{m}(B)$ is of size $m$, so $\E[X^{N_{\A'}^{m}(B)}]\geq 1$. As $N_{\A'}^{m}(B)$ is tame we can apply Lemma \ref{lemma:tameFamily} and get:
$$
\mu_p(N_{\A'}^{m}(B)) \geq \frac{\min\{\E_{p/2}[X^{N_{\A'}^{m}(B)}],1\}}{m2^m}\geq \frac{1}{m2^m}
$$
Finally as $\A'\subset \A$ we get $\mu_p(N_{\A}^{m}(B)) \geq \mu_p(N_{\A'}^{m}(B)) \geq \frac{1}{m2^m}$ as needed.
\end{proof}

Now we are ready to present the structural result and its corollaries.

\begin{theorem} \label{thm:StrTheorem}
 If $\A$ is monotone with minterms of size at most $k$, and let $p \in [0,1]$. Then at least one of the following two possibilities holds.
 \begin{enumerate}
 \item
 There exists a subfamily $\B \subseteq \A$,  with $\mu_{p}(\B) \ge \mu_{p}(\A)/2$, which is tame with respect to $p/2$.
 \item
  There exists $m$ between 1 and $k-1$, and a family $\B$ which is a tame $m$-approximation of $\A$ at $p/2$ with $\mu_{p}(\B) \ge \mu_{p}(\A)/2^{m+1}.$
 \end{enumerate}
\end{theorem}

By induction on the size of the minterms and application of the Theorem \ref{thm:StrTheorem} and corollary 
\ref{cor:TameApp} we deduce the following corollary, whose proof we defer until after the proof of the theorem.

\begin{corollary} \label{cor:p/2top}
Let $\A$ be a monotone family of subsets of $[n]$ with all minterms of size at most $k$ then
$$
\mu_{p/2}(\A)\ge \frac{1}{k2^{3k-1}} \mu_{p}(\A),
$$
\end{corollary}

By repeated application of the corollary above one gets:

\begin{corollary} Let $\A$ be a monotone family of subsets of $[n]$ with minterms of size at most $k$. Then:
$$\delta_\epsilon(\A)\ge 1-(2\epsilon)^\frac{1}{3k+\log k -1}$$
\end{corollary}
Note, again, that this is weaker than what follows from Theorem \ref{thm:decrease}.

\begin{proof}[Proof of Theorem \ref{thm:StrTheorem}]
Let us  iteratively define new families, one of which will be either a tame family or a tame approximation. Let $\A_1 := \A$, for each  $1\leq m \leq k-1$:
\begin{align*}
\B_{m}&=\{V\subset[n] \mbox{ s.t. } |N^{m}_{\A_{m}}(V)|\geq (p/2)^{-m}\}\\
\M(\A_{m+1})&=  \M(\A_{m}) \setminus \{M \mid \exists V\in\B_{m}, V\subset M\}
\end{align*}
and let $\A_{m+1}$ be the family spanned by $\M(\A_{m+1})$.

As $\A=\A_1\supseteq \A_2 \dots \supseteq \A_{k-1} \supseteq \A_{k}$ there are two possible options. Either there is some $m$ for which $\mu_p(\A_m \setminus \A_{m+1})$ is large, or if for all $m$ we have that $\mu_p(\A_m \setminus \A_{m+1})$ is small then $\mu_p(\A_k)$ is large. 

Note that $\A_k$ is tame with respect to $p/2$ as we removed all subsets $V\subset [n]$ that have many supplements (of any size.) If $\mu_{p}(\A_{k})\geq \frac{1}{2} \mu_{p}(\A)$  this gives us immediately the first case of the theorem.

If $\mu_{p}(\A_{k})< \frac{1}{2} \mu_{p}(\A)$ then there must be some $m$ for which $\mu_{p}(\A_{m})-\mu_{p}(\A_{m+1})\geq \frac{1}{2^{m+1}}\mu_{p}(\A)$. Let us show that in this case $\B_{m}$ is a tame $m-$approximation as guaranteed in the second case of the theorem. 

Taking $\A'$ to be $\A_m\subset \A$ we see that from the definition of $\B_m$ one has that for any $B\in \B_m$ $|N_{\A_m}^{m}(B)|> (p/2)^{-m}$ so we only need to show that $N_{\A_m}^{m}(B)$ is tame. 
Indeed, denote $\N$ to be the family spanned by $N^{m}_{\A_m}(B)$ and assume there is some $U\subset [n]$ and some $l<m$ such that $|N_{\N}^{l}(U)|\geq (p/2)^{-l}$, then $|N^{l}_{\A_m}(U\cup B)|\geq |N_{\N}^{l}(U)|\geq (p/2)^{-l}$ and thus $B\cup U \in \B_{l}$ for $l<m$ in contradiction to the definition of $\A_m$.

Finally note that $\mu_{p}(\B_m)=\mu_{p}(\A_{m})-\mu_{p}(\A_m+1)\geq \frac{1}{2^{m+1}}\mu_{p}(\A)$. From this and the above $\B$ is as guaranteed in the second case of the theorem.

\end{proof}

\begin{proof}[Proof of Corollary \ref{cor:p/2top}] 
We will use induction on the size of the minterms. For $k=1$ we note that $\A$ is tame by definition, so we can directly apply Lemma \ref{lemma:tameFamily} and together with the first moment we get the required inequality:
$$
\mu_{p/2}(\A) \geq \frac{\min\{\E_{p/2}[X^{\A}],1\}}{2}\geq \frac{\min\{\E_{p}[X^{\A}],1\}}{4}\geq \frac{1}{4}\mu_{p}(\A)
$$

Now assume we have proved for any $\ell <k$ and let us proof for $k$. Theorem \ref{thm:StrTheorem} gives us two options. If we have the first one, then there is a tame family $\B$ which is a subfamily of $\A$ and $\mu_p(\B)\geq \frac{1}{2} \mu_p(\A)$. Then together with Corollary \ref{cor:TameApp} we have:
$$
\mu_{p/2}(\A)\geq \mu_{p/2}(\B) \geq  \frac{\mu_p(\B)}{k2^k} \geq \frac{\mu_p(\A)}{k2^{k+1}}
$$
which is stronger then the required inequality.

If we are in the second case of the theorem note that any $B\in \M (\B)$ has supplements of size $m$ and so the size of each minterm in $\B$ is at most $k-m$. Thus we can use the induction assumption on $\B$ and get that $ \mu_{p/2}(\B)\geq \frac{1}{(k-m)2^{3(k-m)-1}} \mu_{p}(B) $. 

Recalling that from Theorem \ref{thm:StrTheorem}  $\mu_p{\B}\geq \frac{1}{2^{m+1}}\mu_{p}(\A)$, it is left to apply Corollary \ref{cor:TameApp} and get:
$$
\mu_{p/2}(\A)\geq \frac{1}{m2^m}\mu_{p/2}(\B)\geq \frac{1}{m(k-m)2^{3k-2m-1}} \mu_{p}(B)\geq \mu_{p}(\A)\frac{1}{m(k-m)2^{3k-m}}
$$
a simple calculation will give us the fact that $\frac{1}{m(k-m)2^{3k-m}} \geq \frac{1}{k2^{3k-1}}$ and so we get :
$$
\mu_{p/2}(\A)\geq \frac{1}{k2^{3k-1}}\mu_{p}(\A)
$$
as required
\end{proof}

\section{Fractional expectation and the expectation threshold}\label{sec:LP}
The technique we apply in this section, using duality of linear programming, follows an idea presented by Talagrand in the same context (see "weakly $p$-small'' vs. "$p$-spread'' in \cite{Talagrand}). This leads to a calculation of a weighted-second-moment, as done by Lyons in \cite{Lyons}.

We begin by proving Theorem \ref{thm:LP}. 
\\ {\bf Proof:}  Let $\A \subseteq P([n])$ be a monotone family  with $\M = \M(\A)$ as a set of minterms, 
all of which have size at most $k$. Note that whenever $f : P([n]) \rightarrow \mathbb{R}^{\geq0}$ assumes values greater than 1 on all $A \in \A$ then $\E(f) \geq \mu(\A)$, and whenever $g : P([n]) \rightarrow \mathbb{R}^{\geq0}$ has its support contained in $\A$ then by Payley-Zygmund, 
\beq{PZg}
\mu(\A) \geq \frac {\E(g)^2}{\E(g^2)}.
\enq

 To get good control on 
$\mu(\A)$ it makes sense to try and find such a function $f$ which is as small as possible, and a function $g$ for which the second moment is well behaved (say, not too much weight on the upset generated by any single set, a quantity that arises naturally when calculating the second moment.)
The trick will be to relate these two functions via LP duality.  First, for $q \in [0,1]$, define
$$
{\E}^*_q (\A) =\min \sum \beta(B) q^{|B|}, 
$$
where the minimum is taken over all functions $\beta : P([n]) \rightarrow \mathbf{R}^{\geq 0}$ such that 
\beq{eq:primal}
\sum_{B \subseteq A} \beta(B) \ge 1 \text { for all } A \in \M.
\enq
By LP duality ${\E}^*_q (\A) = L_q^*(\A)$ where
$$
L^*_q (\A) =\max \sum_{A \in \M} \nu(A), 
$$
where the maximum is taken over all functions $\nu : \M(\A) \rightarrow \mathbf{R}^{\geq 0}$ such that 
\beq{eq:dual}
\sum_{B \subseteq A} \nu(A) \leq q^{|B|} \text { for all } B.
\enq
Now, for any $p,q \in (0,1)$ we let $ \alpha = q/p$ and proceed to relate $\mu_p$ and ${\E}^*_q$.
Let  $\nu$ be a function achieving the maximum in the definition of $L^*_q (\A) $ and define 
$$
g(X) = \sum_{A \in \M} \nu(A) p^{-|A|} {\bf 1}_{A \subseteq X}
$$
 and note that 
\beq{eq:numerator}
{\E}_p(g) =  L_q^*(\A). 
\enq
So to complete the calculation of a Payley-Zygmond type lower bound on $\mu_p(A)$ what is left is to calculate
$\E_p(g^2)$. (This is the "weighted second moment" calculation, as in \cite{Lyons}.)  With $A,B$ running over $\M(\A)$ we have
$$
{\E}_p(g^2)  =  \sum \sum \nu(a)\nu(B)p^{- |A \cap B|} 
$$
$$
 \leq \sum_{I \subseteq [n]} p^{-|I|} \left(    \sum_{A \supseteq I} \nu(A)   \right)^2  
$$
 $$
 \leq  \sum_i p^{-i} \left( \max_{|I|=i} \sum_{A \supseteq I} \nu(A) \right)   
  \left(   \sum_{|I|=i} \sum_{A \supseteq I} \nu(A)   \right)     
$$
\beq{eq:defofL}
\sum_i p^{-i}(\alpha p)^i \sum_A \nu(A) \C{|A|}{i}
\enq
$$
\leq \sum_{A} \nu(A) \sum_i \alpha^i \C{|A|}{i}
$$

$$
\sum_{A} \nu(A)(1+\alpha)^{|A|}
$$
\beq{eq:another}
\leq L^*_q(\A) (1+\alpha)^k
\enq
where (\ref{eq:defofL}) follows from (\ref{eq:dual}), and (\ref{eq:another}) follows from the definition of $\nu$ and the fact that all minterms are of size at most $k$.
We now use (\ref{eq:another}) and (\ref{eq:numerator}) in (\ref{PZg}), together with the fact that $L^*={\E}^*$:
$$
\mu_p(\A) \geq \frac {\E(g)^2}{\E(g^2)}  \ge \frac{(L^*_q(\A))^2}{  L^*_q(\A) (1+\alpha)^k} = {\E}^*_{\alpha p } (\A)(1+\alpha)^{-k}.
$$ 
\qed

A nice feature of Theorem \ref{thm:LP} is that it gives sufficient control over the rate of change of $\mu_p$ (as a function of $p$) to give both lower and upper bounds. This is embodied in the following corollary
which also implies theorem\ref{thm:upperbound}.
\begin{corollary}
Let $b < a$. Then 
$$
\left(\frac a b\right)^{1/k} - 1 \leq \frac{p_a}{p_b} \leq \frac {k^k}{(k-1)^{k-1}}\frac {a}{b}.
$$
\end{corollary}
Note that setting $b:=\epsilon, a:=1/2$ implies theorem\ref{thm:upperbound}. Also, note that Theorem 
\ref{thm:decrease} yields the bound $ \left(\frac a b\right)^{1/k}  \leq \frac{p_a}{p_b}$, so that when $a/b$ is large this is almost as good.
\\ {\bf Proof:}
For the lower bound observe that Theorem \ref{thm:LP} implies for any $p$ and $\alpha$
$$
 \frac {\mu_{\alpha p}}{(1+\alpha)^k} \leq \mu_p .
$$
Setting  $p:=p_b$ and $\alpha:= p_a/ p_b$ gives the required result.

For the upper bound it is useful to use the inverse function to ${\E}^*$. Let $q_x$ be the value $q$ for which
${\E}^*_q(\A)=x.$ Theorem \ref{thm:LP} implies for any $x$ and $\alpha$
$$
p_x \le \frac{q_{x(1+\alpha)^k}}{\alpha}.
$$
Also, it is easy to see that for $y \le x$ it holds that $q_x \le q_y \frac{x}{y}$.
Furthermore, for every $x \le 1$ we have $q_x \le p_x$.
Putting these together gives
$$
p_a \le \frac{q_{a(1+\alpha)^k}}{\alpha} \le q_b \frac{(1+\alpha)^k}{\alpha}\frac{a}{b} 
\le p_b \frac{(1+\alpha)^k}{\alpha}\frac{a}{b}.
$$
The function $\frac{(1+\alpha)^k}{\alpha}$ is minimized at $\alpha = 1/(k-1)$. Plugging this value of $\alpha$ into the above expression yields the result.
\qed

\subsection*{Acknowledgements}
The authors wish to thank Oliver Riordan for pointing out to us Theorem \ref{thm:decrease} and its proof.
We would also like to thank Gil Kalai for useful conversations.

\end{document}